\documentclass[10pt, a4paper]{amsart}
\usepackage[utf8]{inputenc}
\usepackage{amsfonts}
\usepackage{amsthm}
\usepackage{amssymb}
\usepackage{amsmath}
\usepackage{newlfont}
\usepackage[mathcal]{euscript}
\usepackage{tikz}
\usepackage{enumerate}
\usepackage{hyperref}
\usepackage{amsthm}
\usetikzlibrary{decorations.pathreplacing}
\usetikzlibrary{patterns}

\usetikzlibrary{arrows,shapes,positioning}
\usetikzlibrary{decorations.markings}
\tikzstyle arrowstyle=[scale=3]
\tikzstyle directed=[postaction={decorate,decoration={markings,
    mark=at position .65 with {\arrow[arrowstyle]{stealth}}}}]
\tikzstyle reverse directed=[postaction={decorate,decoration={markings,
    mark=at position .65 with {\arrowreversed[arrowstyle]{stealth};}}}]
\usepackage{accents}
\DeclareRobustCommand{\ubar}[1]{\underaccent{\bar}{#1}}

 \newtheorem{thm}{Theorem}

\newtheorem{cor}[thm]{Corollary}

\newtheorem*{remark*}{Remark}

\newtheorem{lemma}[thm]{Lemma}
\newtheorem{proposition}[thm]{Proposition}

 \newcommand{\N}{\mathbb{N}}
 \newcommand{\Z}{\mathbb{Z}}

 \def\Fl{F_{l}}
  \def\Ml{M_{l}}  
  \def\Bl{B_{l}}
  \def\A{\mathcal{A}}
  \def\r{\vec{r}}

\def\Ln{\mathcal{L}(n)}

\begin{document}

\title[Topological entropy of the Bunimovich stadium billiard]{An upper bound on topological entropy of the Bunimovich stadium billiard map}

\author{Jernej \v Cin\v c}
\author{Serge Troubetzkoy}

\date{\today}

\subjclass[2020]{37C83, 37B40}
\keywords{Bunimovich stadium billiard, topological entropy.}

\address[J.\ \v{C}in\v{c}]{University of Maribor, Maribor, Koro\v ska 160, Slovenia -- and -- National Supercomputing Centre IT4Innovations IRAFM, University of Ostrava
	30. dubna 22, 70103 Ostrava, Czech Republic}
\email{jernej.cinc@um.si}

\address[S.\ Troubetzkoy]{Aix Marseille Univ, CNRS, I2M, Marseille, France \newline
postal address: I2M, Luminy, Case 907, F-13288 Marseille Cedex 9, France}
\email{serge.troubetzkoy@univ-amu.fr}

\begin{abstract}
We show that  the topological entropy of  the billiard map in a Bunimovich stadium is at most $\log(3.49066)$. 
\end{abstract}
\maketitle
\section{Introduction}
The Bunimovich stadium is  a planar domain whose boundary consists of two
semicircles joined by parallel segments as in Figure \ref{f1}. In this article we study the billiard in a Bunimovich stadium, this is the free motion of a point particle in the interior of the stadium with elastic collisions when the particle reaches the boundary.
Billiards in stadia were first studied by Bunimovich in \cite{B1,B2} where he showed that the billiard has hyperbolic behavior and showed the 
ergodicity, K-mixing and Bernoulli property
of the billiard map and flow with respect to the natural invariant measure  (see also \cite{CM},\cite{CT}).

In this article we will study the topological entropy of the billiard map in a Bunimovich stadium.
The topological entropy of a topological dynamical system is a real nonnegative number that is a measure of the complexity of the system. Roughly, it measures the exponential growth rate of the number of distinguishable orbits as time advances. We will discuss its exact definition in our setting in the next section.

The study of topological entropy of billiards was initiated in \cite{C1}. In this article it was claimed with a one sentence proof that the topological entropy of the billiard map of stadia is at most $\log(4)$.  A detailed proof using this strategy
was given later by B\"aker and Chernov, but they were able to show only a weaker estimate, that the topological entropy is at most $\log(6)$ \cite{BC}. Our main result will be a better upper bound on the topological entropy.

Recently, Misiurewicz and Zhang  \cite{MZ} have shown that as the side length tends to infinity the topological entropy of stadia is at least $\log(1 + \sqrt{2})$ by studying the map restricted to a subspace of the phase space which is compact and invariant under the billiard map. 
Another lower bound of the topological entropy can be derived from the variational principle\footnote{The variational principle holds for the Pesin-Pitskel' definition of entropy \cite{PP}, see Section~\ref{sec5} for applicability to our situation.} and  the results of Chernov on the asymptotics of the  metric entropy when the stadium degenerates to a circle, an infinite stadium, a segment, a point, or the plane in certain controlled ways \cite{Ch97}.

Topological entropy of hyperbolic billiards has also been studied in several other articles \cite{BD},\cite{BFK},\cite{C2},\cite{S}. 

\section{Definitions and statement of the results}\label{sec2}
We consider the Bunimovich stadium billiard table $\Bl$, with the radius of the
semicircles $1$, and the lengths of straight segments $l > 0$. The phase space of
this billiard map will be denoted by $\Ml$. It consists of points $s$ in the boundary of $\Bl$ and unit vectors  pointing into the interior of $\Bl$.
We represent the unit vector by measuring its angle $\theta$ with respect to the inner pointing normal vector, thus $$\Ml  := \{(s,\theta): s \in \partial \Bl, \theta \in (-\pi/2,\pi/2)\}.$$

The billiard map $\Fl$ is the first return map of the billiard flow $\Phi$ to the set $\Ml$.
Note that $\Fl$ is continuous, but $\Ml$ is not compact since we do not include vectors tangent to the boundary of $\Bl$. 

We remark that the map $\Fl$
does not extend to a continuous map of the closure of $\Ml$.
Thus all of the usual definitions of the topological entropy due to Adler, Konheim and McAndrew \cite{AKM},  Bowen \cite{B0,BO1} and Dinaburg \cite{D} can not be applied.
There are several  definitions of topological entropy which are possible.
The definition we take, is a very natural one: we take a natural coding of the billiard, and then consider the entropy of the shift map on the closure of the set of all possible codes. This definition gives an upper bound of another natural definition of topological entropy on non-compact spaces, the Pesin-Pitskel' \cite{PP} topological entropy (this approach is closely related to that of Bowen given in \cite{BO1}, however Bowen's definition is not equivalent to the of Pesin-Pitskel', see \cite{PP}[IV p.308]). In particular, similar results for Sinai billiards (also known as Lorentz gas) were recently obtained by Baladi and Demers \cite{BD}. For a more detailed discussion of possible definitions of topological entropy in our setting and their relationship to our definition see Section \ref{sec5}.

\begin{figure}[ht]
\centering
\begin{tikzpicture}[]

\draw [domain=90:270] plot ({cos(\x)}, {sin(\x)});

\draw [domain=270:450] plot ({5+cos(\x)}, {+sin(\x)});

\draw[] (0,1)--(5,1);
 \draw[] (0,-1)--(5,-1); 
 
 \draw[dotted] (-0.7072,-0.7072) rectangle (5.7071,0.7071); 
 
 \node at (0,1.1){\tiny \color{blue}$b$};
 \node at (0,-1.1){\tiny \color{pink} $p$};
  \node at (5,-1.1){\tiny \color{green} $g$};
 \node at (5,1.1){\tiny \color{red} $r$};

\node at (2.5,-1.2){\tiny $B$};
\node at (2.5,1.2){\tiny $T$};
\node at (-1.2,0){\tiny $L$};
\node at (6.2,0){\tiny $R$};
   
\end{tikzpicture}
\caption{Labeling the sides of the stadium and a period 4 orbit.}\label{f1}
\end{figure}
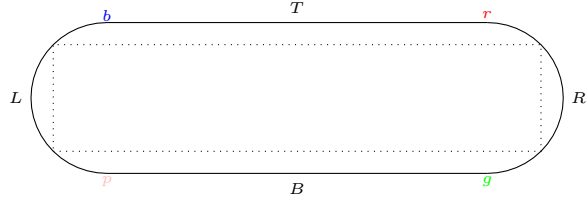

We now give a precise definition of the topological entropy we consider. 
We label the four smooth components of the boundary by the alphabet $\{L,T,R,B\}$, the meeting points of the components have double labels (see Figure \ref{f1}).  Slightly abusing notation we will say that $s\in Y$ where $Y\in \{L,T,R,B\}$ and mean that $s$ is a point in $\partial B_l$ with the label $Y$. It is easy to see that the corresponding partition is not a generator, for example the  period $4$ orbit with code $LLRR$ shown in Figure~\ref{f3} has the same code traced forward and backwards.

 We consider two copies of $L$, denoted by $\bar L$ and $\ubar L$, (similarly $\bar R$ and $\ubar R$ for $R$) and
let $\bar c : \Ml \to \mathcal{A} := \{\bar L ,\ubar L, T, B,\bar  R, \ubar R\}$ be the (multi-valued) {\it coding map} defined by 
$\bar c(s,\theta) = 
s$ if $s \in \{T,B\}$, 
$\bar c(s,\theta) = \bar s$
if $\theta \ge 0$ and $\bar c(s,\theta) =\ubar s$ if $\theta \le 0$ for $s \in \{L,R\}$.
We consider the cover of the phase space into $6$ elements given by this coding.
The interiors of each element of the cover are disjoint, thus
with the traditional misuse of terminology we will call this cover a {\em partition}.

We {\it code} the orbit of a point by the sequence of partition elements it hits, i.e., 
$$c(s,\theta) := (\omega_k)_{k \in \Z} \text{ where } \omega_k = \bar c(\Fl^k (s,\theta)).$$
For $i \le j$ let
 $$\tilde \Ml^{i,j} := \{(s,\theta) \in \Ml: \Fl^n(s,\theta)  \text{ is in the interior of a partition element } \forall i \le n \le j\}.$$
Notice that since $\bar{c}$ is multi-valued, the map $c$ is multi-valued in particular on $\partial \Ml^{i,j}$. However, for any point in the set
$ \tilde  \Ml^{i,j}$  the letter $\omega_k$ is unique for $i \le k \le j$,
and thus for any point in the set
 $$\tilde \Ml := \cap_{i \le j} \tilde  \Ml^{i,j}$$
 the infinite coding is unique.

Let $\tilde{\Sigma}$ be the {\em set of  bi-infinite} codes of points from $\tilde \Ml$, and let $\overline{\Sigma}$
be the closure of $\tilde{\Sigma}$ in the product topology, and let $\Ln$ be the set of words of length $n$ appearing in $\tilde{\Sigma}$ (and thus in $\overline{\Sigma}$ as well).
We let $p(n)$ denote the {\em complexity of $\tilde{\Sigma}$} ; i.e.,
$$p(n) := \#\{(\omega_0,\dots,\omega_{n-1}) \in \Ln\}
.$$

The quantity $ \log  p(n)$ is sub-additive, thus the growth rate 
$$\lim_{n \to \infty} \frac{\log p(n)}n$$
is well defined and is called the {\em topological entropy of the shift map restricted to the set $\overline{\Sigma}$}.

The 6 element partition is a {\em generating partition} in the sense that for each $\omega \in \overline{\Sigma} \setminus \{(TB)^\infty\}$ there is a unique $(s,\theta) \in \Ml$
 whose orbit has code $\omega$ (see \cite{BC} and the Appendix for a justification of this claim)
 thus it is natural to call this quantity the {\em topological entropy of the billiard map $\Fl$},
i.e., $$h_{top}(\Fl) := \lim_{n \to \infty} \frac{\log p(n)}n.$$

In this definition of the topological entropy we first miss a set by restricting to the interiors of partition elements, and then we add some points by taking the closure of $\tilde{\Sigma}$.
The sequences in $\overline{\Sigma} \setminus \tilde{\Sigma}$ are all the codes of points which hit boundaries of partition elements obtained by using one sided continuity extension in the spatial coordinate. 
Although the entropy of $\overline{\Sigma}$ equals the entropy of $\tilde{\Sigma}$, we do not know anything about the Pesin-Pitsel' entropy of the invariant set $\Ml \setminus \tilde{\Ml}$ since the open (clopen) covers of $\overline{\Sigma}$ do not necessarily arise from an open cover of $\Ml$.
In particular we do not know if this entropy is smaller than 
the estimate from Theorem~\ref{t:1}. 

Our work was originally inspired by \cite{MZ} where it was shown that
$$\lim_{l \to \infty} h_{top}(\Fl) \ge \log ( 1 + \sqrt{2}) > \log (2.4142).$$
In fact in \cite{MZ} the authors identify a certain compact subset of the phase space, such that if we restrict $\Fl$ to this set then we get equality in the above limit.

Another inspiration is \cite{BC}; the above mentioned fact about the six element partition being a generating partitions immediately implies  
$$ h_{top}(\Fl) \le \log (6) .$$
In the current paper we improve the upper bound on $h_{top}(\Fl)$.
Let $a := \frac{2 W(\frac{1}{e})}{1 + W(\frac{1}{e})}$ where $W(\frac{1}{e})$ is the unique solution to the equation $1 = w e^{w+1}$,
see \cite{CGHJK} and the beginning of the proof of Lemma~\ref{l7} for more information on the Lambert $W$ function.

The main result of our article is the following theorem

\begin{thm}\label{t:1} For any $l > 0$ we have
$\displaystyle  h_{top}(\Fl) < \log \left ( 2 \left (\frac{2}{a} - 1 \right )^a \right ) < \log (3.49066)$.
\end{thm}

We prove Theorem~\ref{t:1} by studying possible word complexity of the $6$ elements language associated to the Bunimovich billiard. In Section~\ref{s:3} we use Cassaigne's formula from \cite{C} and prove that $h_{top}(\Fl)$ is bounded from above by the limit of logarithmic growth rate of the number of distinct saddle connections of increasing lengths. Cassaigne's formula is very useful in studying low complexity systems, for example polygonal billiards \cite{CHT}. To the best of our knowledge this is the first application of this formula to positive entropy systems.
In Section~\ref{s:4} we give upper bounds for the number of different possible saddle connections using analytical tools,  which yields our estimate for $h_{top}(F_l)$. 

\section{Saddle connections}\label{s:3}

We consider the 6 element partition $\mathcal{A}$ defined in the previous section.
We will use the word {\it corner } to refer to the four points where the semi-circles meet the line segments as well as the two centers of the semi-circles.  More formally, in the case of the centers of the semi-circles, by starting at a corner we mean that we start perpendicularly to a semi-circle and thus the flow passes through the corner when leaving the half-disk defined by the semi-circle. Recall that the four smooth components of the boundary of $B_{l}$ are denoted by the alphabet $\{L,T,R,B\}$, see Figure~\ref{f1}.
The corners separate partition elements, this is clear for the four points, while for the centers of the semicircles we remark 
that the forward and backward orbit of any point $\bar L \cap \ubar L = \{(s,\theta): s \in L, \theta =0 \}$ passes through the center of the left semi-circle (a similar statement holds  for  points in $\bar R \cap \ubar R$). 

In analogy to polygonal billiards a {\it saddle connection} is an orbit segment which connects two corners of $\Bl$ (possibly the same)  and does not visit any corner in between. To avoid technical complications, we do not consider the diameters of the semi-circles as saddle connections. The {\em length of a saddle connection}
is the number of links in this trajectory.  Except for saddle connections of length one we will represent a saddle connection by the first point of collision after leaving the starting corner, thus the code of the orbit segment of length $n-1$ codes the saddle connection. Analogously, the empty word codes the saddle connections with length one.
 Let $N(n)$ denote the \emph{number of distinct saddle connections of length at most $n$}
 and $\mathcal{N}(n)$ denote the \emph{number of distinct saddle connections of length exactly $n$}.
Our main result is based on the following result.
\begin{proposition}\label{t2}
$
p(n)  \le 30 \sum_{j=0}^{n-1} N(j)
$ for all $n \ge 1$.
\end{proposition}

Thus $$h_{top}(\Fl)   = \lim_{n \to \infty}  \frac{\log p(n)}{n} 
\le \lim_{n \to \infty}  \frac{ \log  \big (30   \sum_{j=0}^{n-1} N(j) \big )}{n}  =
 \lim_{n \to \infty}  \frac{ \log  \big (\sum_{j=0}^{n-1} N(j) \big )}{n},$$ which yields

\begin{cor}\label{cor:1}
$$h_{top}(\Fl) \le \lim_{n \to \infty}  \frac{ \log  \big (\sum_{j=0}^{n-1} N(j) \big )}{n} .$$
\end{cor}

To prove the proposition we need some techniques that were developed by Cassaigne in \cite{C} and applied to polygonal billiards in \cite{CHT}.
Remember that $\Ln$ is the {\em set of blocks of length $n$} in the subshift $\overline{\Sigma}$ (so $p(n) = \# \Ln$). For $n \ge 1$, we define $s(n) := p(n+1) - p(n)$. For $u \in \Ln$ let 
\begin{eqnarray*}
m_\ell(u) & := & \# \{a \in \A: au \in \mathcal{L}(n+1)\},\\
m_r(u) & := & \# \{b \in \A: ub \in \mathcal{L}(n+1)\},\\
m_b(u) & := & \# \{(a,b) \in \A^2: aub \in \mathcal{L}(n+2)\}.
\end{eqnarray*}
We remark that all three of these quantities are larger than or equal to one. A word $u \in \Ln$  is called {\it left special} if $m_\ell(u) > 1$, {\it right special }if $m_r(u) > 1$ and {\it bispecial} if it is left and right special. Let $$\mathcal{BL}(n) := \{u \in \Ln: u \text{ is bispecial}\}.$$

In a more general setting in  \cite{C} (see \cite{CHT} for an English version) it was shown that  for all $k \ge 1$ we have
$$s(k+1) - s(k) = \sum_{v \in \mathcal{BL}(k)} \big ( m_b(v) - m_\ell(v) - m_r(v) + 1 \big ).$$

Consider the set of {\em strongly bispecial words} $$\mathcal{BL}_s(n) := \{u \in \Ln: u \text{ is bispecial and }  m_b(v) - m_\ell(v) - m_r(v) + 1  >0\}.$$

Clearly for all $k \ge 1$ we have
$$s(k+1) - s(k) = \sum_{v \in \mathcal{BL}_s(k)} \big ( m_b(v) - m_\ell(v) - m_r(v) + 1 \big ).$$

\begin{proof}[Proof of Proposition \ref{t2}]
Note that 
$$\big ( m_b(v) - m_\ell(v) - m_r(v) + 1 \big ) \le \max_{0 \le x,y \le 6} (xy - x - y +1 ) =25,$$ 
thus
summing over $1 \le k \le j-1$ yields
$$s(j) \le s(1)  +  25 \sum_{k=1}^{j-1} \# \mathcal{BL}_{s}(k).$$
In our case $\# \mathcal{BL}_{s}(1) = p(1) = 6$ and $p(2) = 30$
since among the 36 possible words, the 6 words that can not be realized are 
$TT,BB,\bar{L}\ubar{L},\ubar{L}\bar{L},\bar{R}\ubar{R},\ubar{R}\bar{R}$.
Thus 
$s(1) = p(2) - p(1) = 30 - 6 = 24 = 4 \#\mathcal{BL}_{s}(1)$, and thus we can estimate

$$s(j) \le   29 \sum_{k=1}^{j-1} \# \mathcal{BL}_{s}(k).$$
Remember that $s(j) = p(j+1) - p(j)$, thus summing over $ 1 \le j \le n-1$ yields

$$p(n) \le p(1)  +  29  \sum_{j=1}^{n-1} \sum_{k=1}^{j-1} \# \mathcal{BL}_{s}(k).$$
Again we can adjust the constant to absorb the term $p(1)$ yielding the  estimate
$$
p(n) \le  30 \sum_{j=1}^{n-1} \sum_{k=1}^{j-1} \# \mathcal{BL}_{s}(k).$$
To finish the proof of Proposition~\ref{t2} we need part  a) from  the following result.

\begin{proposition}\label{p3}
 a) For each $k \ge 1$ there is an injection 
 $\mathcal{C}: \mathcal{BL}_{s}(k-1) \to \mathcal{N}(k)$.\\
b) For any $v \in \mathcal{BL}(k)$ and any pair of corners $(s,s')$ there is at most one saddle connection with code $v$ starting at the corner $s$ and ending at the corner $s'$.

\end{proposition}
   Once a) is proven, this yields
$$  
p(n) \le  30 \sum_{j=1}^{n-1} \sum_{k=1}^{j-1} \# \mathcal{N}(k+1) \le 30 \sum_{j=0}^{n-1} N(j).
$$
which completes the proof of Proposition \ref{t2}. \end{proof}
 
\begin{figure}[ht]\label{ff}
\centering
\begin{tikzpicture}[scale=0.7, reflect at x/.style={xshift=#1*1cm,xscale=-1,xshift=-1*#1*1cm},
reflect at y/.style={yshift=#1*1cm,yscale=-1,yshift=-1*#1*1cm}]

\draw (0,0) -- (16,0);
\draw (0,5) -- (16,5);
\draw[blue, thick, dotted] (0,0) -- (0,5);
\draw[blue, thick, dotted] (16,0) -- (16,5);

\draw[dash dot] (0,2.5) -- (2,2.5);
\draw[dash dot] (8,2.5) -- (10,2.5);

\draw[pink, dotted, thick] (2,0) -- (2,5);
\draw[green, dotted, thick] (8,0) -- (8,5);
\draw[red, dotted, thick] (10,0) -- (10,5);

\node at (0,-0.2){\tiny \color{blue} $b$};
\node at (2,-0.2){\tiny \color{pink} $p$};
\node at (8,-0.2){\tiny \color{green} $g$};
\node at (10,-0.2){\tiny\color{red}  $r$};
\node at (16,-0.2){\tiny \color{blue}$b$};

\node at (1,-0.3){\tiny $L$};
\node at (5,-0.3){\tiny $B$};
\node at (9,-0.3){\tiny $R$};
\node at (13,-0.3){\tiny $T$};

\node at (-0.4,2.5){\tiny \color{red} $0$};
\node at (-0.4,0){\tiny \color{red} $ \text{-}\frac{\pi}{2}$};
\node at (-0.4,5){\tiny \color{red} $\frac{\pi}{2}$};

\draw [thick,red,dashed,reflect at y=2.5] (8,2.5) -- (10,0);
\draw [thick,red,dashed, reflect at y=2.5] (2,1) -- (8,2.5);
\draw [thick,red,dashed, reflect at y=2.5]  (0,5) -- (2,1);

\draw [thick,blue,dashed, reflect at y=2.5] (8,4) -- (10,0);
\draw [thick,blue,dashed, reflect at y=2.5] (2,2.5) -- (8,4);
\draw [thick,blue,dashed, reflect at y=2.5]  (0,5) -- (2,2.5);

\draw[thick,pink,dashed, reflect at y=2.5] (0,2.5) -- (2,0);
 \draw[thick,pink,dashed, reflect at y=2.5] (10,1) --(16,2.5);
 \draw[thick,pink,dashed,reflect at y=2.5] (8,5) -- (10,1);

\draw [thick,green,dashed, reflect at y=2.5] (0,4) -- (2,0);
 \draw[thick,green,dashed, reflect at y=2.5] (10,2.5) --(16,4);
 \draw[thick,green,dashed, reflect at y=2.5] (8,5) -- (10,2.5);
\end{tikzpicture}
\caption{The singularity sets of $F_l$ (dotted) and $F_l^{-1}$  (dashed). The singularity sets are monotone curves but for clarity they are drawn as linear segments.
The two segments which are both dotted and dashed are in both singularity sets.}\label{f2}
\end{figure}
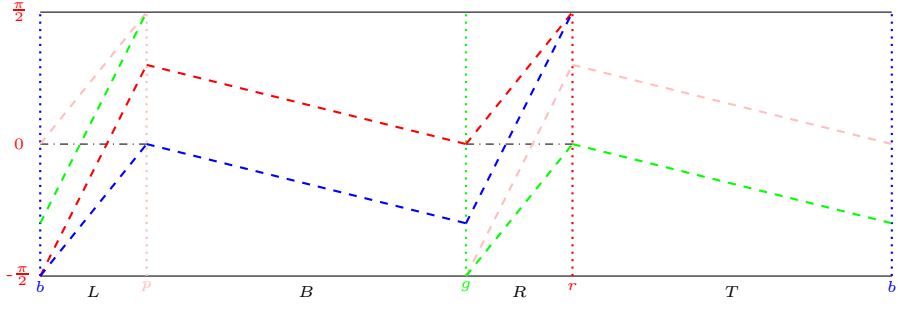
 
 Before proving Proposition~\ref{p3} we need to introduce some more terminology.
 Let $\Gamma'$ be the {\em set of points from $M_l$ perpendicular to the semicircles}, i.e. $\Gamma'\subset M_l$ is the set of points for which $s\in L,R$ and  $\theta=0$.
 Let $\Gamma$ be the {\em union of $\Gamma'$ with the set of points where $F_l$ fails to be $C^2$}, and
 analogously $\Gamma^{-}$ is the {\em union of $\Gamma'$ with the set of points where $F_l^{-1}$ fails to be $C^2$}. We call $\Gamma$ and $\Gamma^-$ the {\em singularity sets for $F_l$ and $F_l^{-1}$} respectively.
 
 The singularity sets $\Gamma,\Gamma^-$ consist of a finite number of $C^1$-smooth compact curves in $\Ml$ (see Figure \ref{ff}), which are increasing, decreasing, horizontal, or vertical.
 Define the {\em singularity set for the map $F_l^n$} for $n\geq 1 $ by $\Gamma^n:=\bigcup^{n}_{i=1} F_l^{-i+1}(\Gamma)$ and the {\em singularity set of the map $F_l^{-n}$} for $n\geq 1$ by $\Gamma^{-n}:=\bigcup^{n}_{i=1} F_l^{i+1}(\Gamma^{-})$.

Remember that the set $\tilde \Ml$ defined in Section \ref{sec2} is the set of points in $\Ml$ having a well defined code, and 
 $\tilde \Ml$ is the set of points whose orbit does not hit a corner.

For $(s,\theta) \in \tilde \Ml$  let $c_k(s,\theta) 
:= (\bar c(\Fl^i (s,\theta))_{i=0}^{k-1}$ denote the {\em block of length $k$ containing 
 $c(s,\theta)$}.
 For $v \in \mathcal{L}(k)$ we define the set $$\omega(v) := \overline{\{(s,\theta) \in \tilde \Ml: v = c_k(s,\theta)\}}$$ and call $\omega(v)$ a {\it $k$-cell}.

\begin{proof}[Proof of Proposition \ref{p3}]
 a)  For $k=1$ we note that the empty word is bispecial and it corresponds to  22 saddle connections, one for each pair of distinct corners excluding diameters and sides, and thus a) holds for $k=1$.

Now suppose $k \ge 2$ and fix $v \in \mathcal{BL}_{s}(k-1)$. The proof of Lemma 2.5 in \cite{BC} shows  that the set $\omega(v)$ is a simply connected closed set whose boundary consists of a finite
collection of piecewise smooth curves with angles less than $\pi$ at vertices. These curves belong to the union of the singular sets of $F_l^{i}$  for $0 \le i \le k-2$.
For each $0 \le i \le k-2$ the map $F_l^i$ is continuous on $\omega(v)$.

Consider the  ``partition''
 $ \bigcup_{avb \in \mathcal{L}(k+1)} F_l (\omega(avb))$ of the set $\omega (v)$, this is a partition in the sense that the interiors of the partition elements are pairwise disjoint.
 This partition  is produced by cutting
 $\omega(v)$ by the   singular sets of $F_l^{k-1}$ and $F_l^{-1}$.

By assumption $v$ is bispecial, so 
the branches of the singular set of $F_l^{k-1}$ cut 
$\omega(v)$ into $m_r(v) \ge 2$ pieces and the  branches  of the singular set of $F_l^{-1}$  cut $\omega(v)$ into $m_\ell(v) \ge 2$
pieces.
 Suppose first that these singularities do not intersect,
 then the union of these singular sets cut $\omega(v)$
into $m_r(v) + m_\ell(v) -1$ pieces and thus the word $v$ is weakly bispecial, i.e., $v \in \mathcal{BL}(n) \setminus \mathcal{BL}_s(n)$ and as mentioned above does not contribute to the sum, see Figure~\ref{fig:singularities} left.

\begin{figure}[h!]
	\begin{tikzpicture}[scale=3]
	\draw[thick] (0,0)--(1,0)--(0.8,0.3)--(0.1,0.4)--(0,0);
	\draw[red,thick] (0.05,0.2)--(0.5,0.34);
	\draw[blue,thick] (0.7,0)--(0.7,0.31);
	\end{tikzpicture}
 	\hspace{1cm}
 \begin{tikzpicture}[scale=3]
	\draw[thick] (0,0)--(1,0)--(0.8,0.3)--(0.1,0.4)--(0,0);
	\draw[red,thick] (0.05,0.2)--(0.87,0.2);
	\draw[blue,thick] (0.7,0)--(0.7,0.31);
	\end{tikzpicture}
 \caption{ Examples of $\omega(v)$ which are weakly and strongly bispecial with $m_\ell(v)= m_r(v) = 2$.}\label{fig:singularities}
\end{figure}
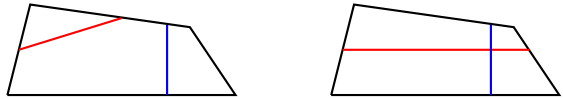

Consider a point $x$ of intersection of these two singular 
 sets. As mentioned above the angle formed is less than $\pi$, i.e.,
   the intersection must be transverse.   If this intersection is on the boundary of a cell, then the orbit of $x$ has at least three singular collisions, as so by definition $x$ does not 
 represent a saddle connection. So suppose $x$ is  in the interior of a cell  $\omega(v)$. We have the preimage of $x$ is a corner, and its forward image by $F_l^{k-1}$ is also a corner, and
 all intermediate collisions are non-singular, thus it 
corresponds to a saddle connection of length ${k-1}$.
 Thus  for $k \ge 2$ we have verified part a).

 We turn to the verification of part b).
Label the corners of $\Bl$ by the alphabet $\mathcal{A'} := \{1,2,3,4,5,6\}$.  The {\it code of a saddle connection} is the sequence from $\mathcal{A} \cup \mathcal{A}'$ a point hits along with the starting and ending corners; thus a saddle connection of length $n$ will have a code of
length $n+1$.
To finish the proof we need to show that there is a bijection between saddle connections and their codes.

We will give a brief 
sketch describing the bijection between the set of codes and possible trajectories of the billiard map. 
A smooth curve from the phase space $\Ml$ equipped with a continuous family of unit normal vectors is called a {\em wave front}.
Suppose by way of contradiction that two trajectories start at the same corner $s'$ and end at the corner $s''$ (possibly $s''=s'$) and have the same code. For concreteness
the starting points are $(s',\theta_1)$ and $(s',\theta_2)$.   We consider the 
set $G:= \{(s',\theta): \theta_1 \le \theta\le \theta_2\}$ and for each $t \ge 0$ the corresponding wave front $G_t = \Phi{_t}(G)$.
 A wave front is said to {\em focus} at time $t > 0$ when the projection of the wave front $G_t$ to the billiard table  intersects itself.
By way of contradiction we thus assumed the wave front $G_t$
refocuses at a corner; 
we will show that this is
in fact impossible.
We refer to Subsection 8.4. in \cite{CM} for a more complete description of what follows.

 Focusing occurs in Bunimovich billiards  after the wave front reflects from one of the two semi-circles.
Suppose that an infinitesimal wave front $G_t$ collides with $\partial \Ml$ at some point  in a  semi-circle; denote 
the  post-collisional curvature of the projection of $G_t$ to the billiard table by $\mathcal{G}^+_t$.
The curvature of a wave front does not change at the instance of a collision with a flat boundary.
Now suppose that the projection of $G_t$ to the billiard table experiences collisions with semi-circles at times $t$ and $t+\tau$,  with possibly some flat collisions in between.
Using (3.35) from \cite{CM} the wave front expands from a collision to another collision, if $|1+\tau \mathcal{G}^+_t|>1$. For this to hold, it is enough to check that  $\mathcal{G}^+_{t}<-2/\tau$ (see (8.2) from \cite{CM}).
A focusing wave front with curvature $\mathcal{G}^+_{t}<0$ passes through a focusing point and defocuses at the time $t^*=t-1/\mathcal{G}^+_{t}$ or in other words $\mathcal{G}^+_t=\frac{1}{t-t^*}$ (see Section 3.8 in \cite{CM}).  Thus $\mathcal{G}^+_t<-2/\tau$ is equivalent to $t^*<t+\tau/2$.  The last inequality indeed says that the wave front must defocus before it reaches the midpoint between the consecutive collisions.
By Theorem 8.9. from \cite{CM} it holds that the families of unstable cones remain unstable under the iteration of the map. This implies that all wave fronts are in the unstable cones which gives a contradiction. Therefore, we indeed have unique coding of trajectories of the billiard map.
\end{proof}

\section{Proof of Theorem \ref{t:1}}\label{s:4}

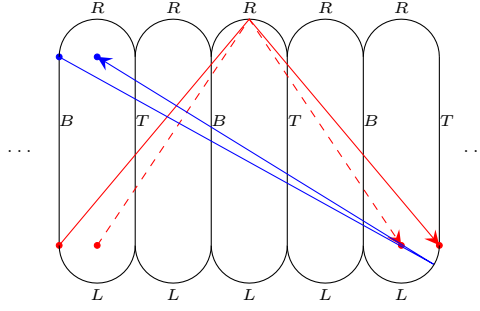
\begin{figure}[ht]
\centering
\begin{tikzpicture}[scale=0.5, rotate=90]

\draw [domain=90:270] plot ({cos(\x)}, {sin(\x)});

\draw [domain=270:450] plot ({5+cos(\x)}, {+sin(\x)});

\draw[] (0,5)--(5,5);
\draw[] (0,3)--(5,3);
\draw[] (0,1)--(5,1);
\draw[] (0,-1)--(5,-1); 
\draw[] (0,-3)--(5,-3);
\draw[] (0,-5)--(5,-5);

\node at (3.3,-5.2){\tiny $T$};
\node at (3.3,-3.2){\tiny $B$};
\node at (3.3,-1.2){\tiny $T$};
\node at (3.3,0.8){\tiny $B$};
\node at (3.3,2.8){\tiny $T$};
\node at (3.3,4.8){\tiny $B$};

\node at (6.3,-4){\tiny $R$};
\node at (6.3,-2){\tiny $R$};
\node at (6.3,0){\tiny $R$};
\node at (6.3,2){\tiny $R$};
\node at (6.3,4){\tiny $R$};

\node at (-1.3,-4){\tiny $L$};
\node at (-1.3,-2){\tiny $L$};
\node at (-1.3,0){\tiny $L$};
\node at (-1.3,2){\tiny $L$};
\node at (-1.3,4){\tiny $L$};

\draw[red, decoration={markings,mark=at position 1 with
    {\arrow[scale=1.7,>=stealth]{>}}},postaction={decorate}] (0,5) -- (6,0) -- (0,-5);
\draw[red,fill=red] (0,5) circle (.5ex);
\draw[red,fill=red] (0,-5) circle (.5ex);
\draw[red, dashed, decoration={markings,mark=at position 1 with
    {\arrow[scale=1.7,>=stealth]{>}}},postaction={decorate}] (0,4) -- (6,0) -- (0,-4);
\draw[red,fill=red] (0,4) circle (.5ex);
\draw[red,fill=red] (0,-4) circle (.5ex);
\draw[blue, decoration={markings,mark=at position 1 with
    {\arrow[scale=1.7,>=stealth]{>}}},postaction={decorate}] (5,5) -- (-0.5,-4.85) -- (5,4);
\draw[blue,fill=blue] (5,5) circle (.5ex);
\draw[blue,fill=blue] (5,4) circle (.5ex);

\draw [domain=90:270] plot ({cos(\x)},  {2+sin(\x)});
\draw [domain=270:450] plot ({5+cos(\x)}, {2+sin(\x)});

\draw [domain=90:270] plot ({cos(\x)},  {-2+sin(\x)});
\draw [domain=270:450] plot ({5+cos(\x)}, {-2+sin(\x)});

\draw [domain=90:270] plot ({cos(\x)},  {4+sin(\x)});
\draw [domain=270:450] plot ({5+cos(\x)}, {4+sin(\x)});

\draw [domain=90:270] plot ({cos(\x)},  {-4+sin(\x)});
\draw [domain=270:450] plot ({5+cos(\x)}, {-4+sin(\x)});

\node at (2.5, 6){\tiny $\dots$};
\node at (2.5, -6){\tiny $\dots$};
\end{tikzpicture}
\caption{Unfolding the stadium. Remember that centers of semi-circles are also corners.
There are eight saddle connections with signed composition $2,1,2$, two of them having 
 code $TB\bar RTB$ are drawn in red, there are two more saddle connections with this code, the other four  have code $TB\bar LTB$.
If we reverse the arrows we obtain the saddle connection with signed composition $-2,1,-2$.
In blue we show a saddle connection with 
code $ TBTB{\ubar L}BTBT$ 
and signed composition $4,1,-4$. If we reverse the arrows the blue saddle connection has the same code and signed composition.}\label{f3}
\end{figure}

To count the number of saddle connections we unfold the stadium (see Figure \ref{f3}).
Consider an integer $j \ge 1$.  We say $(n_1, m_1, n_2, m_2, \dots n_k , m_k)$ is a {\it signed composition} of $j$
if $n_i \in \Z$, $m_i \in \N$ with $m_i \ge 1$ for all $i$ such that $\sum |n_i| + \sum m_i = j$.
Let $Q(j)$ denote the number of {\em signed compositions of $j$}. Recall that $N(n)$ denotes the number of distinct saddle connections of length at most $n$.

\begin{lemma}
For each $n$ and $l > 0$ we have
$$  N(n) \le 36 \sum_{j=0}^{n-1}  Q(j).$$
\end{lemma}

\begin{proof}
Fix a corner of $\Bl$ and consider the saddle connections of length at most $n$ starting at this corner.  
We consider the associated signed composition in the following way:  the non-negative integer  $|n_i|$ counts the consecutive hits in the flat sides of $\Bl$ 
and $m_i$ counts the consecutive hits in a semicircle.  The sign of $n_i$ tells us which way we are moving in the unfolding, left or right, when changing from one semicircle to the other.
In this way each saddle connection yields a signed composition.  

Fix $j\geq 1$ and a signed composition of $j$.
As we showed in  part b) from Proposition \ref{p3}, for each pair of corners there is at most one saddle connection with this signed composition. Thus, since there are 6 corners,
 there are at most 36 codes of saddle connections which correspond to a given signed composition.
\end{proof}

 Let $(n_1, m_1, n_2, m_2, \dots n_k , m_k)$ be a signed composition of $j$ with $2k$ terms. Denote by $Q(j,k)$ the {\em number of such possible compositions of $j\in \N$ with $2k$ terms}. 
 Let
 $$r_i := |n_i| + m_i, \text{ then } \r_k(j) := (r_1, \dots , r_k)$$ 
 is a {\em composition of $j$ with $k$ terms}.
In what follows we will first estimate $Q(j,k)$ and then $Q(j)$.

 Fix a composition $\r_k(j)$.  Each $n_i \in \{-r_i+1,  \dots, -1,0, 1, \dots r_i-1\}$ yields a different signed composition,  there are 
 $$f(\r_k(j)) := \prod_{i=1}^k (2 r_i- 1)$$  preimages of  $\r_k(j)$ in total, i.e., 
 $$Q(j,k) =\sum_{\ell \ge 1} \ell \times  \# \{\r_k(j): f(\r_k(j)) = \ell  \}.$$
 We start by estimating the number of terms in this sum, i.e., the largest possible value of $\ell$. If $s=q_1+\ldots +q_k$, then the arithmetic-geometric mean inequality
 $$\sqrt[k]{q_1\cdots q_k}\leq \frac{q_1+\ldots +q_k}{k}$$ 
yields $$q_1\cdots q_k\leq \left (\frac{s}{k} \right )^k.$$ 
Notice that  the equality is obtained if and only if all the $q_1 = q_2 = \dots = q_k$, and thus $q_i = s/k$.
Setting $q_i = 2r_i -1$ and $s = 2j-k$ yields
   $$ f(\r_k(j)) \le \left (\frac{2j}{k}-1 \right )^k$$
with equality if and only if $\displaystyle r_i = \frac{2j}{k} -1$.

Thus
\begin{align*}
Q(j,k) & \le \left (\frac{2j}{k}-1 \right )^k \sum_{\ell \ge 1}  \# \Big \{\r_k(j): f(\r_k(j))
 = \ell  \Big \} \\ &= \left (\frac{2j}{k}-1 \right )^k \times \# \Big \{\r_k(j) \Big \}\\
 &= \left ( \frac{2j}{k}-1 \right )^k  \binom{j}{k}.
\end{align*}
Fix $j \ge 1$ and let $g_j$ be the function
defined by
$$k \in \{1, \dots, j\} \mapsto \left (\frac{2j}{k}-1 \right )^k  \binom{j}{k}.$$

\begin{lemma}\label{l:comb}
The function $k \in \{1,\dots,j\} \mapsto \binom{j}{k}$ is increasing for $1 \le k \le \frac{j+1}{2}$ and decreasing for $\frac{j+1}{2} \le k \le j$.
\end{lemma}

\begin{proof}
The inequalities  $\displaystyle 1 \le k-1<k \le \frac{j+1}{2}$ imply that $\displaystyle \frac{j-k+1}{k} \ge 1$ and thus
$$ \binom{j}{k} = \binom{j}{k-1}\frac{j-k+1}{k} \ge \binom{j}{k-1}$$
and thus the function is increasing for $1 \le k \le \frac{j+1}{2}$.

The decreasing statement holds since $\binom{j}{k} = \binom{j}{j-k}$.
\end{proof}

For each $1 \le j$ let $$h_j(x) = \left (\frac{2j}{x}-1  \right )^x = e^{x\ln(\frac{2j}{x}-1)}.$$ 
\begin{lemma}\label{l6}
For each $j \ge 2$ there exists a unique $x_j>1$
such that $h_j$ is increasing for $x\in [1,x_j]$ and decreasing for $x\in [x_j,j]$.
\end{lemma}

\begin{proof}
Throughout the proof the functions under consideration are restricted to the domain $[1,j]$.
We begin by calculating the derivative of $h_j$,
$$h'_j(x) = h_j(x)  \left ( \frac{-2j}{ 2j-x } + \ln \left (\frac{2j}{x}-1 \right )\right ).$$
Let $\displaystyle k_j(x) := \frac{-2j}{ 2j-x } + \ln \left (\frac{2j}{x}-1 \right )$, then the signs of $h'_j$ and $k_j$ are the same since $h_j$ is positive. We study the sign of $k_j$ by taking its derivative: 
$$k'_j(x) =
-\frac{2j}{\left(2j-x\right)^2}-\frac{2j}{x\left(-x+2j\right)}
=
\frac{-4j^2}{x(2j-x)^2} < 0$$ 
for $x \in [1,j]$.

For $j \ge 2$ we have $k_j(1) =  \frac{-2j}{ 2j-1 } + \ln \left ({2j}-1 \right ) > 0$.
Furthermore, $k_j(j) = -2 < 0$;
thus there is a unique $x_j \in (1, j)$ which is the solution of the equation $k_j(x) = 0$ such that $\mathrm{sgn}(h'_j(x)) = \mathrm{sgn}(k_j(x)) > 0$ for $x \in (1,x_j)$ and 
$\mathrm{sgn}(h'_j(x)) = \mathrm{sgn}(k_j(x)) < 0$ for $x \in (x_j, j)$ and thus
 $h_j(x)$ is maximized when $x=x_j$. 
\end{proof}

To prove the next lemma we will use the Lambert $W$ function; see \cite{CGHJK} for an introduction to this notion. The Lambert $W$ function is a multivalued function which for a given complex number $z$ gives all the complex numbers $w$  which satisfy the equation $w e^w =z$. If $z$ is a positive real number then there is a single real solution $w$ of this equation which we denote $W(z)$.

\begin{lemma}\label{l7}
There exists a constant $a$ such that $x_j = a \cdot j$ for each $j \ge 2$.
\end{lemma}
\begin{proof}
The equation $k_j(x_j) = 0$ is equivalent to 
$
    \frac{2j - x}{x}  =  e^{1}e^{\left (\frac{x}{2j-x}\right )}.
$
   Substituting $w = \frac{x}{2j-x}$ yields  
$     1/w  =  e^1 e^w$
or equivalently
$\frac{1}{e} = w e^w.
$
Since $\frac{1}{e}$ is positive there is a single solution to this equation $w = W(\frac{1}{e})$, and thus
$x_j = \frac{2 W(\frac{1}{e})}{1 + W(\frac{1}{e})} \cdot j =: a \cdot j$.
\end{proof}

\begin{lemma}\label{l:max}
The constant $a$ verifies $a \in (0.43562,0.43563)$.
The maximum value of $h_j$ is at most
$\left ( \frac{2}{a} - 1 \right )^{aj} < 1.74533^j$.
\end{lemma}
\begin{proof}
Notice that
$k_j(0.43562 j) \approx  0.00001 > 0$
and
$k_j(0.43563 j) \approx -0.00002< 0$.
Remembering from Lemma~\ref{l6} that $k_j$ is decreasing yields
 $0.43562 < a < 0.43563$. 
The maximum value of $h_j$ is
$\displaystyle h_j(x_j) = h_j(a \cdot j) = 
\left  ( \frac{2}{a} - 1 \right )^{aj}$.
If $0< a_1 < a$
 then
$\displaystyle  
 \frac{2}{a} - 1  <  
 \frac{2}{a_1} - 1.$
If furthermore $a < \min(a_2,1)$ then $\frac{2}{a} - 1 > 1$ and thus
$\displaystyle  
\left  ( \frac{2}{a} - 1 \right )^a  <  
\left  ( \frac{2}{a_1} - 1 \right )^{a_2}.$ 
Combining this with our previous estimate yields
 $\left  ( \frac{2}{a} - 1 \right )^a < \left  ( \frac{2}{0.43562} - 1 \right )^{0.43563} < 1.74533$.
\end{proof}

Combining Lemma~\ref{l:comb} and Lemma~\ref{l:max}  yields:
\begin{cor}\label{cor:estimate}
For  $j \ge 2$  the
maximum value of function $g_j$ is
bounded from above by
  $\displaystyle
  \left ( \frac{2}{a} - 1 \right )^a \binom{j}{\lfloor \frac{j}{2} \rfloor} < 1.74533^j \binom{j}{\lfloor \frac{j}{2} \rfloor}.$
\end{cor}

To prove the next result we consider the gamma function $\Gamma(z)$, we will use
the Legendre duplication formula
$$\Gamma(z) \Gamma(z+ \frac{1}{2}) = 2^{1-2z} \sqrt{\pi} \Gamma(2z)$$
as well as Gautschi's
inequality
$$x^{1-s} < \frac{\Gamma(x+1)}{\Gamma(x+s)} < (x+1)^{1-s}$$
which holds for any positive real $x$ and $s \in (0,1).$
\begin{lemma}\label{lem:estimate}
For any even $j \ge 2$ we have $$\binom{j}{\lfloor \frac{j}{2} \rfloor} \le 
\sqrt{\frac{2}{j} }\cdot \frac{2^{j}}{\sqrt{\pi}}$$
while for odd $j > 2$ we have
$$\binom{j}{\lfloor \frac{j}{2} \rfloor} \le
\sqrt{\frac{2}{j+1}}\cdot \frac{2^{j}}{\sqrt{\pi}}.$$\end{lemma}
\begin{proof}
If $j=2n$ is even then
$\displaystyle \binom{j}{\lfloor \frac{j}{2} \rfloor} =
\binom{2n}{n} =
\frac{\Gamma(2n+1)}{\Gamma(n + 1)^2}$.
Using the duplication formula with $z = n + \frac{1}{2}$ yields
$$\displaystyle 
\frac{\Gamma(2n+1)}{\Gamma(n + 1)^2} = \frac{\Gamma(2z)}{\Gamma(z + \frac{1}{2})^2}
= \frac{\Gamma(z)}{\Gamma(z+ \frac{1}{2})}\cdot \frac{2^{2z-1}}{\sqrt{\pi}} = \frac{\Gamma(\frac{j+1}{2})}{\Gamma(\frac{j}{2}+ 1)} \cdot \frac{2^{j}}{\sqrt{\pi}}.$$

Next we apply Gautschi's inequality with $s = \frac{1}{2}$ and $x = \frac{j}{2}$; it yields 
$$\frac{\Gamma(\frac{j+1}{2})}{\Gamma(\frac{j}{2}+ 1)}\cdot \frac{2^{j}}{\sqrt{\pi}} < \sqrt{ \frac{2}{j} }\cdot \frac{2^{j}}{\sqrt{\pi}}.$$

Now suppose that $j=2n+1$ is odd, then
$$\displaystyle \binom{j}{\lfloor \frac{j}{2} \rfloor} =
\binom{2n+1}{n} =
\frac{\Gamma(2n+2)}{\Gamma(n+1)\Gamma(n + 2)}.$$
Using the duplication formula with $z = n + 1$ yields
$$\displaystyle 
\frac{\Gamma(2n+2)}{\Gamma(n + 1)\Gamma(n+2)} = \frac{\Gamma(2z)}{\Gamma(z)\Gamma(z+1)}
= \frac{\Gamma(z+ \frac{1}{2})}{\Gamma(z+ 1)}\cdot \frac{2^{2z-1}}{\sqrt{\pi}} = \frac{\Gamma(\frac{j}{2}+1)}{\Gamma(\frac{j+1}{2}+1)} \cdot \frac{2^{j}}{\sqrt{\pi}}.$$

Again we apply Gautschi's inequality, here with $x = \frac{j+1}{2}$ and $s = \frac{1}{2}$ which yields
$$\frac{\Gamma(\frac{j}{2}+1)}{\Gamma(\frac{j+1}{2}+1)} \cdot \frac{2^{j}}{\sqrt{\pi}}< \sqrt{\frac{2}{j+1}}\cdot \frac{2^{j}}{\sqrt{\pi}}$$
\end{proof}

Now we are ready to prove the main theorem of the paper.

\begin{proof}[Proof of Theorem~\ref{t:1}]
From Corollary~\ref{cor:estimate} and Lemma~\ref{lem:estimate} it follows that

$$Q(j,k)\le \left( \left  ( \frac{2}{a} - 1 \right )^a\right )^j \binom{j}{\lfloor \frac{j}{2} \rfloor} \le \left( \left  ( \frac{2}{a} - 1 \right )^a\right )^j \frac{2^j}{\sqrt{j\pi/2}}= \frac{\left( 2\left  ( \frac{2}{a} - 1 \right )^a\right )^{j}}{\sqrt{j\pi/2}}.$$

Therefore,
   $$Q(j) \le j \max_k(Q(j,k)) \le   
   \frac{\left(2 \left  ( \frac{2}{a} - 1 \right )^a\right )^{j} \sqrt{j}}{\sqrt{\pi/2}} 
  .$$
  
  Thus we obtain
  $$N(n) \le 36 \sum_{j=1}^{n-1} Q(j) \le 36 \sum_{j=1}^{n-1}
    \frac{\left(2 \left  ( \frac{2}{a} - 1 \right )^a\right )^{j} \sqrt{j}}{\sqrt{\pi/2}} 
    \le  {\left(2 \left  ( \frac{2}{a} - 1 \right )^a\right )^{n}} C \sqrt{n-1},
  $$
  where $C$ is a positive constant. Recall that $p(n)$ denotes the complexity of $\tilde{\Sigma}$. Therefore using Proposition \ref{t2}, 
  $$p(n)\le 30\sum_{j=0}^{n-1} {\left(2 \left  ( \frac{2}{a} - 1 \right )^a\right )^{j}} C \sqrt{j-1}\le {\left(2 \left  ( \frac{2}{a} - 1 \right )^a\right )^{n}} C' \sqrt{n-1}$$ where $C'$ is another positive constant.
 From the definition of topological entropy we obtain $$h_{top}(\Fl)\le \log \left ( 2 \left  ( \frac{2}{a} - 1 \right )^a \right ) \le \log(3.49066).$$
\end{proof}

\section{Other possible definitions of topological entropy}\label{sec5}
Another very natural definition of 
 topological entropy was 
 given by Pesin and Pitskel' in \cite{PP} and the closely related
 capacity topological entropy was defined by Pesin in \cite{P}[Page 75].
 Applying these definitions to the map
 $F_l$ restricted to $\tilde{M}_l$
 yields two quantities, the Pesin-Pitskel' topological entropy $h_{\tilde{\Ml}}(\Fl)$ and the capacity topological entropy $Ch_{\tilde \Ml}(\Fl)$.
 Formally, in \cite{P} the capacity topological entropy is defined   in a slightly more restrictive setting than in \cite{PP}, but it can be defined in the setting of \cite{PP} and the relationship $h_{\tilde \Ml}(\Fl) \le Ch_{\tilde \Ml}(\Fl)$ from \cite{P} still holds.
 But our definition of $h_{top}(\Fl)$ coincides with Pesin's definition of $Ch_{\tilde \Ml}(\Fl)$. Thus we have the following corollary
 
 \begin{cor} For any $l > 0$ 
 the Pesin-Pitskel' topological entropy $h_{\tilde \Ml}(\Fl)$
 is bounded from above by 
$\log \left ( 2 \left (\frac{2}{a} - 1 \right )^a \right ) < \log (3.49066)$.
\end{cor}

\section*{Appendix}
 
In \cite{BC} the authors state in Theorem 3.5 that  a certain 16 element partition is generating.  Just after the statement of the theorem they remark that it  implies that the six element partition we consider in this article is a generating partition as well.

For convenience we give a formal proof of this remark.
Consider the four element partition $\mathcal{B} := \{L,T,B,R\}$. The
sixteen element partition $\mathcal{C}$ consists of the connected components of the partition
$\mathcal{C}' := \mathcal{B} \vee \Fl^{-1}(\mathcal{B})$. The partition $\mathcal{C}'$ has 14 elements since the codes $TT$ and $BB$ can not be realized.
There are exactly two  elements of $\mathcal{C}'$ which are not connected, corresponding to  
 the pairs $LL$ and $RR$. Each of them has two connected components, 
 we call the splitting into the two connected components $\overline{LL},\underline{LL},$ resp.~$\overline{RR},\underline{RR}$ which yields the partition $\mathcal{C}$.  The component $\overline{LL}$ ($\overline{RR}$) consists of $x \in \Ml$ such that $x$ and $\Fl(x)$ are on the left (right) semi-circle of $\partial B_l$ and
$x = (s,\theta)$ with $\theta \le 0$. Similarly, the component $\underline{LL}$   ($\underline{RR}$) consist of such points where $\theta \ge 0$.

We consider the space of all codes $\tilde{\Sigma}_6$ resp.\ $\tilde{\Sigma}_{16}$ of orbits which stay in the interiors of the partition elements with the six letter, resp.\ 16 letter alphabet, and their 
closures
$\overline{\Sigma}_6$ resp.\ $\overline{\Sigma}_{16}$. 
Note that $\tilde{\Sigma}_6$ resp.\ $\overline{\Sigma}_6$ is referred to $\tilde{\Sigma}$ resp.\ $\overline{\Sigma}$ in the rest of this article.
Theorem 3.5 of \cite{BC} states that the partition $\mathcal{C}$ is a generating partition in the 
following sense:
there is a continuous surjection  
$$\pi : \tilde{\Sigma}_{16} \setminus ((TB)^\infty \cup (BT)^{\infty}) \to \tilde \Ml \setminus \hat \Ml$$
such that 
$$\Fl \circ \pi = \pi \circ \sigma$$ where $\sigma$ is the shift map
and 
$\hat \Ml := \{x \in \Ml: c(s,\theta) = (BT)^\infty \text{ or } (TB)^\infty\}.$

Actually $\mathcal{C}$ is a generating partition in a stronger sense; namely
there is a continuous surjection  

$$\pi : \overline{\Sigma}_{16} \setminus ((TB)^\infty \cup (BT)^{\infty}) \to  \Ml \setminus \hat \Ml$$
such that $\Fl \circ \pi = \pi \circ \sigma$.
The proof of
this stronger statement is identical to the proof of Theorem 3.5 of \cite{BC}, and it is important to note that in Equation (10) of their proof the authors show that the
intersection of closed cells is a single point.

We define $\phi: \overline{\Sigma}_6 \to \overline{\Sigma}_{16}$
by regrouping consecutive symbols and erasing the overbar or underbar except for the four cases
$\overline{LL},\underline{LL},\overline{RR},\underline{RR}$.

For example in Figure \ref{f1} we consider the lower point on $L$ pointing up, we have 
$$\phi(\dots,\ubar{L},\ubar{L},\ubar{R},\ubar{R}, \dots) = \dots,
\underline{LL}, LR, \underline{RR}, RL, \dots$$ 
while if we consider the upper point on $L$ pointing down, we have
$$\phi(\dots,\bar{L},\bar{L},\bar{R},\bar{R}, \dots) = \dots,
\overline{LL}, LR, \overline{RR}, RL, \dots.$$ 

The map $\phi$ is continuous, surjective, and commutes with the shift map, thus 
$$\pi \circ \phi : \overline{\Sigma}_{6} \setminus ((TB)^\infty \cup (BT)^{\infty}) \to  \Ml \setminus \hat \Ml
$$
is a continuous surjection and 
such that $\Fl \circ \pi \circ \phi =  \pi \circ \phi \circ \sigma$,
i.e., the six element partition is generating as well.

\section*{Acknowledgements and statements}
We thank Peter B\'alint, Lyonia Bunimovich, Micha{\l } Misiurewicz and Yasha Pesin for 
useful discussions and helpful remarks. We also thank the referees for detailed 
constructive comments.
J. \v Cin\v c was partially supported by the FWF Schr\"odinger Fellowship stand-alone project J 4276-N35, the IDUB program no. 1484 ``Excellence initiative - research university'' for the AGH University of Science and Technology and the AD Program J1-4632 from ARRS, Slovenian National Research Agency.

Data sharing is not applicable to this article as no datasets were generated or analysed during the current study.
Both authors declare that they have no conflicts of interest.

\end{document}